\documentclass[12pt]{article} 

\usepackage[margin=3cm]{geometry}
\usepackage[english]{babel}
\usepackage{hyperref}
\usepackage{amsmath}
\usepackage{amsthm}
\usepackage{amssymb}
\usepackage{enumerate}
\usepackage[utf8]{inputenc}
\usepackage[T1]{fontenc}
 \usepackage{authblk}
 
\usepackage{tikz}
\usetikzlibrary{arrows}
\usetikzlibrary{decorations.markings}
\pgfdeclarelayer{nodelayer}
\pgfdeclarelayer{edgelayer}
\pgfsetlayers{nodelayer,main,edgelayer}
\tikzstyle{none}=[draw=none]   
\tikzstyle{bigtiparrow}=[->,thick, >=angle 90]
\tikzstyle{bigtiparrow2}=[->,thick, >=angle 90,preaction={draw=white, -,line width=6pt}]
\tikzstyle{lrarrow}=[<->,thick, >=angle 90,preaction={draw=white, -,line width=6pt}]

\tikzstyle{new}=[rectangle,fill=white,draw=white, inner sep=2pt]
\tikzstyle{new2}=[rectangle,fill=white,draw=white, inner sep=6pt]

\DeclareMathOperator{\rng}{\mathrm{rng}}
\DeclareMathOperator{\supp}{\mathrm{supp}}
\DeclareMathOperator{\Cost}{\mathrm{Cost}}

\newcommand{\Stild}{\widetilde{\mathfrak S}}

\newcommand{\MAlg}{\mathrm{MAlg}}
\newcommand{\Aut}{\mathrm{Aut}}

  \newcommand{\N}{\mathbb N}
  
  \newcommand{\Z}{\mathbb Z}

 \newcommand{\Orb}{\mathrm{Orb}}    
 
 \newcommand{\dom}{\mathrm{dom}\;}
  
    \newcommand{\Stab}{\mathrm{Stab}}

  \newcommand{\inv}{^{-1}}

   \newcommand{\Sub}{\mathrm{Sub}}
  
  \renewcommand{\leq}{\leqslant}
  \renewcommand{\geq}{\geqslant}
  \newcommand{\abs}[1]{\left\lvert #1\right\rvert}

  \newcommand{\act}{\curvearrowright}
  
  \newcommand{\impl}{\Rightarrow}

  \newcommand{\la}{\left\langle}
  \newcommand{\ra}{\right\rangle}
  \newcommand{\into}{\hookrightarrow}

\newtheorem{thm}{Theorem}[section]
\newtheorem{cor}[thm]{Corollary}
\newtheorem{lem}[thm]{Lemma}
\newtheorem{prop}[thm]{Proposition}
\newtheorem{qu}{Question}

\theoremstyle{definition}
\newtheorem*{claim}{Claim}

\newtheorem{df}[thm]{Definition}
\newtheorem*{rmq}{Remark}
\newtheorem*{ack}{Aknowledgments}
\newtheorem{exemple}[thm]{Example}
\renewcommand*{\thefootnote}{\fnsymbol{footnote}}

\title{Highly faithful actions and dense free subgroups in full groups}
\author[1]{François Le Maître\thanks{Research supported by Projet ANR-14- CE25-0004 GAMME.}}

\date{\vspace{-5ex}}
\begin{document}

\setcounter{footnote}{1}

\maketitle
\renewcommand*{\thefootnote}{\arabic{footnote}}

\begin{abstract}
In this paper, we show that every measure-preserving ergodic equivalence relation of cost less than $m$ comes from a ‘‘rich'' faithful invariant random subgroup of the free group on $m$ generators, strengthening a result of Bowen which had been obtained by a Baire category argument. 

Our proof is completely explicit: we use our previous construction of  topological generators for full groups and observe that these generators induce a totally non free action. We then twist this construction so that the action is moreover amenable onto almost every orbit and highly faithful. 

In particular, we obtain that the full group of a measure-preserving ergodic equivalence of cost less than $m$ contains a dense free subgroup on $m$ generators.
\end{abstract}

\tableofcontents

\section{Introduction}

A natural conjugacy-invariant for a measure-preserving action of a countable group $\Gamma$ on a standard probability space $(X,\mu)$ is the associated \textit{measure-preserving equivalence relation} $\mathcal R_\Gamma$ defined by $(x,y)\in\mathcal R_\Gamma$ if and only if $\Gamma x=\Gamma y$. 
Such equivalence relations are studied up to \textit{orbit equivalence}, that is up to isomorphism and restrictions to full measure sets.

Measure-preserving actions of countable groups are often asked to be \textit{free}: every non-trivial group element fixes almost no point. The study of free measure-preserving actions up to orbit equivalence is well developed and has fruitful connections to measured group theory and von Neumann algebras, see \cite{MR2827853} for a recent overview. 

The theory of \textit{cost}, introduced by Levitt  \cite{MR1366313} and developed by Gaboriau \cite{MR1728876}, has proven to be invaluable in this area. The cost of a measure-preserving equivalence relation $\mathcal R$ is the infimum of the measures of its generating sets\footnote{See section \ref{sec:prelim} for a precise definition.}, thus providing an analogue of the rank of a countable group for measure-preserving equivalence relations. A fundamental theorem due to Gaboriau is that every free action of the free group on $m$ generators induces a measure-preserving equivalence relation of cost $m$ (see \cite[Cor. 1]{MR1728876}). 

In this paper, we are interested in non-free actions of $\mathbb F_m$. Our starting point is the contrapositive of Gaboriau's aforementioned theorem: a measure preserving equivalence relation of  cost less than $m$ cannot come from a free action of $\mathbb F_m$. Moreover, an easy consequence of Gaboriau's results is that \textit{every} ergodic measure-preserving equivalence relation of cost less than $m$ comes from a non-free action of $\mathbb F_m$.

It is then natural to search for some strengthening of non-freeness for $\mathbb F_m$-actions so as to further classify measure-preserving equivalence relations of cost less than $m$. We thus ask:

\begin{qu}\label{qu:mainqu}
Consider a measure-preserving ergodic equivalence relation of cost less than $m$. How non-free can the $\mathbb F_m$-actions that induce it be?
\end{qu}

We now list three ways a measure-preserving action of the free group on $m\geq 2$ generators can be thought of as ‘‘very'' non-free. 
\subsection{Non freeness I: Amenability onto almost every orbit}

\begin{df}\label{df:amenable action}An action of a countable group $\Gamma$ on a set $Y$ is called \textbf{amenable} if it admits a sequence of almost invariant sets, i.e. if there exists a sequence of finite subsets $(F_n)$ of $Y$ such that for all $\gamma\in\Gamma$, 
$$\frac{\abs{\gamma F_n\bigtriangleup F_n}}{\abs{F_n}}\to 0\,[n\to+\infty].$$
A countable group $\Gamma$ is \textbf{amenable} if its left action onto itself by translation is amenable.
\end{df}
The group of integers $\Z$ is a key example of an amenable group (the sequence of intervals $[-n,n]$ is almost invariant). On the other hand, for any $n\geq 2$ the free group $\mathbb F_n$ is \textit{not} amenable: for instance, one can build a \textit{Ponzi scheme} on it (see \cite[Cor. 6.18]{MR1699320}). 

\begin{df}
A measure-preserving action of a countable group $\Gamma$ on a standard probability space $(X,\mu)$ is called \textbf{amenable onto almost every orbit} if for almost every $x\in X$, the $\Gamma$-action on $\Gamma\cdot x$ is amenable.
\end{df}

\begin{exemple}
Suppose that $\Gamma\act(X,\mu)$ is a free measure-preserving action. Then for almost every $x\in X$, the $\Gamma$-equivariant map $\Gamma\to \Gamma\cdot x$ which takes $\gamma\in\Gamma$ to $\gamma\cdot x$ is a bijection. Hence almost all the actions on the orbits are conjugate to the left $\Gamma$-action onto itselft by translation. So a free $\Gamma$-action on $(X,\mu)$ is amenable onto almost every orbit if and only if $\Gamma$ is amenable. \end{exemple}

We deduce from the previous example that for all $n\geq 2$,  a free measure-preserving action of the free group $\mathbb F_n$ is never amenable onto almost every orbit. In particular, measure-preserving actions of $\mathbb F_n$ which \textit{are} amenable onto almost every orbit can be thought of as very non-free actions. Examples of non-amenable measure-preserving equivalence relations coming from  $\mathbb F_m$-actions which are amenable onto almost every orbit were first constructed by Kaimanovich \cite{MR1485618}.

\subsection{Non freeness II: High transitivity onto almost every orbit}

\begin{df}Let $\Gamma$ be a countable group acting on a set $Y$. The action is \textbf{highly transitive} if for every $n\in\N$, the diagonal $\Gamma$-action on the set of $n$-tuples made of pairwise distinct elements of $Y$ is transitive.
 
To be more precise , the action is highly transitive if for every $n\in\N$, every pairwise distinct $y_1,...,y_n\in Y$ and every pairwise distinct $y'_1,...,y'_n\in Y$, there exists $\gamma\in\Gamma$ such that for all $i\in\{1,...,n\}$, we have $\gamma\cdot y_i=y'_i$. 
\end{df}
As an example, the natural action of the group of finitely supported permutations of the integers is highly transitive. It has been an ongoing research theme to understand which countable groups admit faithful highly transitive actions; see \cite{Hull:2015fr} for a striking recent result in that area.

It is a well-known fact that a permutation group $\Gamma\leq\mathfrak S(Y)$ is highly transitive if and only if it is dense for the topology of pointwise convergence. Note that a nontrivial highly transitive action can never be free. The following definition was introduced by Eisenmann and Glasner \cite{Eisenmann:2014oq} and can   also be seen as a strengthening of non-freeness for measure-preserving actions.

\begin{df}Let $\mathcal R$ be a measure-preserving equivalence relation on a standard probability space $(X,\mu)$. A measure-preserving action of a countable group $\Gamma$ on $(X,\mu)$ is \textbf{almost surely highly transitive on $\mathcal R$-classes} if for almost every $x\in X$, $\Gamma$ preserves the equivalence class $[x]_{\mathcal R}$ and acts on it in a highly transitive manner.
\end{df}

There is a very nice sufficient condition for a group to act almost surely highly transitively on $\mathcal R$-classes and to state it we need to introduce full groups.

\begin{df}
Let $\mathcal R$ be a measure-preserving equivalence relation. Its \textbf{full group}, denoted by $[\mathcal R]$, is the group of all measure-preserving Borel bijections $T$ of $(X,\mu)$ such that for all $x\in X$, we have $T(x)\in[x]_{\mathcal R}$. Moreover, two such bijections are identified if they coincide up to a null set.
\end{df}

Whenever $T$ and $U$ are measure-preserving bijections of $(X,\mu)$, one can define the uniform distance between them by 
$$d_u(T,U):=\mu(\{x\in X: T(x)\neq U(x)\}).$$

Whenever $\mathcal R$ is a measure-preserving equivalence relation, the uniform metric induces a complete separable metric on its full group which is thus a Polish group (see e.g. \cite[Prop. 3.2]{MR2583950}).

We may now state Eisenmann and Glasner's result.

\begin{thm}[{\cite[Prop. 1.19]{Eisenmann:2014oq}}]\label{thm:dense implies ht}
    Let $\mathcal R$ be a measure-preserving equivalence relation. Let $\Gamma\leq[\mathcal R]$ be a countable dense subgroup of $[\mathcal R]$. Then $\Gamma$ acts almost surely highly transitively on $\mathcal R$-classes.
\end{thm}

It is not true in general that any almost surely highly transitive action comes from a dense embedding into a full group\footnote{To see this, start with $\Gamma\act (X,\mu)$ which is almost surely highly transitive on $\mathcal R$-classes. Then consider the $\Gamma$-action on two disjoint copies of $(X,\mu)$ and let $\mathcal R'$ be the associated equivalence relation. The new action is almost surely highly transitive on $\mathcal R'$-classes, but $\Gamma$ is not dense in $[\mathcal R']$ since any element of the closed subgroup generated by $\Gamma$ has to act the same on the two copies of $(X,\mu)$.}. However in the ergodic case the question of the converse was asked by Eisenmann and Glasner. 

\subsection{Non freeness III: Total non freeness}

To state properly one last possible definition for an action to be very non free, we need to introduce invariant random subgroups, which are important invariants of non-free measure-preserving actions.

Let $\Gamma$ be a countable group. We denote by $\Sub(\Gamma)\subseteq\{0,1\}^\Gamma$ the space of closed subgroups of $\Gamma$, which is a closed subspace of the compact metrizable space $\{0,1\}^\Gamma$ equipped with the product topology. 
With the induced topology, $\Sub(\Gamma)$ is thus a compact metrizable space naturally acted upon by $\Gamma$ via conjugacy: for any $\Lambda\in\Sub(\Gamma)$ and any $\gamma\in\Gamma$, one lets $\gamma\cdot\Lambda:=\gamma\Lambda\gamma\inv$.

\begin{df}
An \textbf{invariant random subgroup} (or IRS) of a countable group $\Gamma$ is a $\Gamma$-invariant Borel probability measure on $\Sub(\Gamma)$.
\end{df}

Let $\Gamma\act(X,\mu)$ be a measure-preserving action. The map $\Stab: X\to \Sub(\Gamma)$ which maps $x\in X$ to $\Stab_\Gamma(x)$ is $\Gamma$-equivariant, so by pushing forward the measure $\mu$ we obtain an IRS $\Stab_*\mu$ of $\Gamma$. Abert, Glasner and Virag have shown that the converse is true: every IRS of $\Gamma$ can be written as $\Stab_*\mu$ for some measure-preserving $\Gamma$-action on $(X,\mu)$ \cite[Prop. 13]{MR3165420}.

\begin{df}[Vershik]
Let $\Gamma\act(X,\mu)$ be a measure-preserving action. It is called \textbf{totally non free} if the map $\Stab:(X,\mu)\to(\Sub(\Gamma),\Stab_*\mu)$ is a conjugacy.
\end{df}

Note that since the map $\Stab$ is $\Gamma$-equivariant, the only thing one has to check in order to know that an action is totally non free is that $\Stab$ becomes injective when restricted to a suitable full measure subset of $X$. In the setting of full groups, our observation is the following.

\begin{prop}[see Proposition \ref{prop:dense implies tnf}]\label{prop:dense implies tnf1}
Let $\mathcal R$ be a measure-preserving aperiodic\footnote{A measure-preserving equivalence relation is aperiodic if almost all its classes are infinite.} equivalence relation. If $\Gamma\leq[\mathcal R]$ is a dense countable subgroup, then the $\Gamma$-action is totally non free.
\end{prop}

Bowen obtained a satisfactory answer to Question \ref{qu:mainqu} in the context of totally non-free actions: he showed by a Baire category argument that whenever $\mathcal R$ is an ergodic equivalence relation of cost less than $n$, there exists a totally non free action of the free group on $n$ generators which induces the equivalence relation $\mathcal R$ \cite{MR3420547}. We remark that this result can also be obtained by combining the previous proposition with \cite[Thm. 1]{gentopergo}.

\subsection{Statement of the main result}

Our main result is that the above conditions for non-freeness can be achieved all at once along with high faithfulness. The latter is a strengthening of the notion of faithfulness and is somehow dual to high transitivity (see Section \ref{sec: highlyfaith} for more on this notion; our definition differs significantly from the one given by Fima, Moon and Stalder in \cite{MR3411128}).

\begin{df}A transitive action of a countable group $\Gamma$ on a set $Y$ is \textbf{highly faithful} if for all $n\in\N$ and all pairwise distinct $\gamma_1,...,\gamma_n\in\Gamma$, there exists $y\in Y$ such that for all distinct $i,j\in\{1,...,n\}$, we have $\gamma_i\cdot y\neq\gamma_j\cdot y$. 
\end{df}

Note that the natural action of the group of finitely supported permutations of the integers is highly transitive faithful, but not highly faithful. It would be interesting to understand which countable groups admit highly faithful highly transitive actions. 

A measure-preserving action of a countable group is called highly faithful if it is highly faithful onto almost every orbit. Here this notion was useful to us in order to obtain sequences of sets with nice disjointness properties (see item (\ref{cond:nice disjoint sets for all Fn}) in Theorem \ref{thm:cara highly faithful pmp}) and also to produce (highly) faithful actions for some free products via Theorem \ref{thm:highly faithful rf}. We can now state our main result, which upgrades \cite[Thm. 1]{gentopergo}.

\begin{thm}\label{thm:main thm}Let $\mathcal R$ be an ergodic equivalence relation with finite cost. Then for all $m\in\N$ such that $m>\Cost(\mathcal R)$ there is a dense free group on $m$ generators in the full group of $\mathcal R$ whose action is moreover  amenable onto almost every orbit and highly faithful.
\end{thm}

As a corollary, we can strengthen Bowen's Theorem and generalize a result that Eisenmann and Glasner had obtained for cost 1 ergodic measure-preserving equivalence relations by a Baire category argument \cite[Cor. 21]{Eisenmann:2014oq}.

\begin{cor}\label{thm:main cor}Let $\mathcal R$ be an ergodic equivalence relation with finite cost. Then for all $m\in\N$ such that $m>\Cost(\mathcal R)$ there is a totally non-free highly faithful action of $\mathbb F_m$ which induces the equivalence relation $\mathcal R$ and which is highly transitive and amenable onto almost every orbit.
\end{cor}
\begin{proof}
By Proposition \ref{prop:dense implies tnf1} and Theorem \ref{thm:dense implies ht}, the $\mathbb F_m$-action obtained via Theorem \ref{thm:main thm} is totally non free and highly transitive onto almost every orbit. Since it is also highly faithful and amenable onto almost every orbit, we are done.
\end{proof}

Since total non-freeness implies that the stabiliser map is an isomorphism, the above result implies the following statement about invariant random subgroups (see \cite{Eisenmann:2014oq} for the definitions of the terms used thereafter). 

\begin{cor}\label{thm:main cor IRS}Let $\mathcal R$ be an ergodic equivalence relation with finite cost. Then for all $m\in\N$ such that $m>\Cost(\mathcal R)$ there is an IRS of $\mathbb F_m$ which induces the equivalence relation $\mathcal R$ and which is core-free, co-highly transitive and co-amenable.
\end{cor}

All the above results admit non-ergodic counterparts where we require $\mathcal R$ to be aperiodic and  its conditional cost  to be almost surely less than $m$. However supposing that the equivalence relation is ergodic makes proofs much lighter and we hope this will help convey the ideas of this work. The interested reader will be able to ‘‘convert'' the proofs presented here to their non-ergodic analogues by a careful reading of \cite{lm14nonerg}.

We now give an outline of this paper. The next section is devoted to notation and the proof of Proposition \ref{prop:dense implies tnf1}. In Section \ref{sec: highlyfaith} we introduce and study high faithfulness. In Section \ref{sec: rf groups and hf} we build highly faithful actions of free products $\Gamma*\Lambda$ where $\Gamma$ already acts highly faithfully and $\Lambda$ is any residually finite group. Section \ref{sec:top gen hyperfinite} is devoted to a flexible construction of topological generators for the full group of the hyperfinite ergodic equivalence relation. Theorem \ref{thm:main thm} is finally proven in Section \ref{sec:pf of main thm}.

\begin{ack} I would like to thank the anonymous referee for her or his helpful comments.
\end{ack}

\section{Preliminaries}\label{sec:prelim}

Let $(X,\mu)$ be a standard probability space. We will always work modulo sets of measure zero. Let us first briefly review some notation and definitions.

We denote by $\Aut(X,\mu)$ the group of all measure-preserving Borel bijections of $(X,\mu)$. Given $T\in\Aut(X,\mu)$, its \textbf{support} is the set $\supp T:=\{x\in X: T(x)\neq x\}$.

Let $A$  and $B$ be Borel subsets of $X$, a \textbf{partial isomorphism} of $(X,\mu)$ of \textbf{domain} $A$ and \textbf{range} $B$ is a Borel bijection $f: A\to B$ which is measure-preserving for the measures induced by $\mu$ on $A$ and $B$ respectively. A \textbf{graphing} is a countable set $\Phi=\{\varphi_1,...,\varphi_k,...\}$ where the $\varphi_k$'s are  partial isomorphisms. It \textbf{generates} a \textbf{measure-preserving equivalence relation} $\mathcal R_\Phi$, defined to be the smallest equivalence relation containing the graphs of the partial isomorphisms belonging to $\Phi$.  The \textbf{cost} of a graphing $\Phi$ is the sum of the measures of the domains of the partial isomorphisms it contains. The \textbf{cost} of a measure-preserving equivalence relation $\mathcal R$ is the infimum of the costs of the graphings that generate it, we denote it by $\Cost(\mathcal R)$. The cost of $\mathcal R$ is \textbf{attained} if there exists a graphing $\Phi$ which generates $\mathcal R$ such that $\Cost(\Phi)=\Cost(\mathcal R)$. We refer the reader to the lectures notes by Gaboriau\footnote{These are available online at \url{http://perso.ens-lyon.fr/gaboriau/Travaux-Publi/Copenhagen/Copenhagen-Lectures.html}.} for an efficient overview of cost theory. 

The \textbf{full group} of $\mathcal R$ is the group $[\mathcal R]$ of automorphisms of $(X,\mu)$ which preserve the $\mathcal R$-classes, that is
$$[\mathcal R]=\{\varphi\in\Aut(X,\mu): \forall x\in X, \varphi(x)\,\mathcal R\, x\}.$$
It is a Polish group when equipped with the complete biinvariant metric $d_u$ defined by $$d_u(T,U)=\mu(\{x\in X: T(x)\neq U(x)\}.$$ One also defines the \textbf{pseudo full group} of $\mathcal R$, denoted by $[[\mathcal R]]$, which consists of all partial isomorphisms $\varphi$ such that $\varphi(x)\, \mathcal R \, x$ for all $x\in \dom\varphi$.

Let $p\in\N$. A \textbf{pre}-$p$-\textbf{cycle} is a graphing $\Phi=\{\varphi_1,...,\varphi_{p-1}\}$ such that  the following two conditions  are satisfied:
\begin{enumerate}[(i)]\item  $\forall i\in\{1,...,p-2\}, \rng\varphi_i=\dom\varphi_{i+1}$.
\item The following sets are all disjoint:
 $$\dom\varphi_1, \dom \varphi_2,...,\dom\varphi_{p-1},\rng\varphi_{p-1}.$$
\end{enumerate}
A $p$\textbf{-cycle} is an element $C\in \Aut(X,\mu)$ whose orbits have cardinality $1$ or $p$. 

Given a pre-$p$-cycle $\Phi=\{\varphi_1,...,\varphi_{p-1}\}$, we can extend it to a $p$-cycle $C_\Phi\in\Aut(X,\mu)$ as follows:
$$C_\Phi(x)=\left\{\begin{array}{ll}\varphi_i(x) & \text{if }x\in \dom\varphi_i\text{ for some }i<p, \\
\varphi_1\inv\varphi_2\inv\cdots\varphi_{p-1}\inv(x) &\text{if }x\in \rng\varphi_{p-1},\\
x & \text{otherwise.}\end{array}\right.$$

Say that a measure-preserving equivalence relation $\mathcal R$ is \textbf{ergodic} when every Borel $\mathcal R$-saturated set has measure $0$ or $1$. The following standard fact about ergodic measure-preserving equivalence relations is the main source of pre-$p$-cycles, and hence of $p$-cycles.

\begin{prop}[see e.g. \cite{MR2095154}, lemma 7.10.]\label{isopar}
Let $\mathcal R$ be an ergodic measure-preserving equivalence relation on $(X,\mu)$, let $A$ and $B$ be two Borel subsets of $X$ such that $\mu(A)=\mu(B)$. Then there exists $\varphi\in[[\mathcal R]]$ of domain $A$ and range $B$.
\end{prop}

The following theorem is fundamental for building dense subgroups of full groups.
\begin{thm}[\cite{MR2599891}, Thm. 4.7]\label{ktdense}
Let $\mathcal R_1$, $\mathcal R_2$,... be measure-preserving equivalence relations on $(X,\mu)$, and let $\mathcal R$ be their join (i.e. the smallest equivalence relation containing all of them). Then $\la\bigcup_{n\in\N}[\mathcal R_n]\ra$ is dense in $[\mathcal R]$.
\end{thm}

An easy application is the following proposition. 

\begin{prop}[{\cite[Prop. 10]{gentopergo}}]\label{isopgen}
If $\Phi=\{\varphi_1,...,\varphi_{p-1}\}$ is a pre-$p$-cycle, then for all $i\in\{1,...,p-1\}$, the full group of $\mathcal R_\Phi$ is topologically generated by $[\mathcal R_{\{\varphi_i\}}]\cup\{C_\Phi\}$.
\end{prop}

Let us finally turn to the relationship between dense subgroups of full groups and total non freeness that we mentioned in the introduction.

\begin{prop}\label{prop:dense implies tnf}
Let $\mathcal R$ be a measure-preserving aperiodic equivalence relation. If $\Gamma\leq[\mathcal R]$ is a dense countable subgroup, then the $\Gamma$-action is totally non free.
\end{prop}
\begin{proof}
We will work in the setting of measure algebras\footnote{See \cite[Chap. 2]{MR1958753}
 for some background on measure algebras.}: to see that $\Stab: (X,\mu)\to(\Sub(\Gamma),\Stab_*\mu)$ is a bijection between full measure sets, it suffices to show that the injective map $\Stab\inv: \MAlg(\Sub(\Gamma),\Stab_*\mu)\to \MAlg(X,\mu)$ is also surjective. Moreover since its image is closed, it suffices to show that its image is dense.

For all $\gamma\in\Gamma$, let $A_\gamma:=\{\Lambda\in\Sub(\Gamma): \gamma\not\in \Lambda\}$. Then  for all $\gamma\in\Gamma$, we have $\Stab\inv(A_\gamma)=\supp(\gamma)$. So it suffices to show that the family $(\supp(\gamma))_{\gamma\in\Gamma}$ is dense in $\MAlg(X,\mu)$. But this follows from the density of $\Gamma$ in $[\mathcal R]$ and the well-known fact that for all $A\in\MAlg(X,\mu)$, there exists $T\in[\mathcal R]$ such that $\supp(T)=A$.
\end{proof}

\section{Highly faithful actions}\label{sec: highlyfaith}

Let us now study in details the notion of a highly faithful action. 

\begin{df}An action of a countable group $\Gamma$ on a countable set $Y$ is called $n$-faithful if for any $\gamma_1,...,\gamma_n\in\Gamma\setminus\{1\}$, there exists $y\in Y$ such that $\gamma_i y\neq y$ for all $i=1,...,n$. The action is \textbf{highly faithful} if it is $n$-faithful for every $n\in\N$; in other words if for any finite subset $F\subseteq \Gamma\setminus\{1\}$, there exists $y\in Y$ such that $fy\neq y$ for all $f\in F$.
\end{df}

Note that every free action is highly faithful, and that an action is faithful iff it is $1$-faithful. A simple example of a faithful action of an infinite group which is not highly faithful is given by the group  $\mathfrak S_{(\infty)}$ of finitely supported permutations of $\N$ acting on $\N$. Note that this action is however highly transitive. I don't know if $\mathfrak S_{(\infty)}$ can have a highly transitive highly faithful action.

\begin{lem}\label{lem:highly faithful is locally free}Let $\Gamma$ be a countable group acting on a set $Y$. Then the action is highly faithful iff for all $n\in\N$ and all pairwise distinct $\gamma_1,...,\gamma_n\in\Gamma$, there exists $y\in Y$ such that for all distinct $i,j\in\{1,...,n\}$, 
$$\gamma_iy\neq\gamma_jy.$$
\end{lem}
\begin{proof}
Apply the definition of high faithfulness to the finite set $F:=\{\gamma_i\gamma_j\inv: i\neq j\in \{1,...,n\}\}$.
\end{proof}

The previous lemma has the following nice geometric interpretation when $\Gamma$ is a finitely generated group: a transitive action is highly faithful if and only if the associated Schreier graph contains arbitrarily large balls of the Cayley graph of $(\Gamma,S)$ for some (or any) finite generating set $S$. 

In this article our focus will be on the measured version of high faithfulness. 

\begin{df}
A measure-preserving action of a countable group $\Gamma$ on a probability space $(X,\mu)$ is called \textbf{highly faithful} if for almost every $x\in X$,  the $\Gamma$-action on $\Gamma\cdot x$ is highly faithful. \qedhere
\end{df}

We will now give a useful characterization of highly faithful actions. The proof uses the following well-known lemma.

\begin{lem}[{see e.g. \cite[Lem. 5.1]{Eisenmann:2014oq}}] \label{lem: disjoint image}
Let $T\in \Aut(X,\mu)$, let $A\subseteq X$ such that $\mu(\{x\in A: T(x)\neq x\})>0$. Then there exists a positive measure set $A'\subseteq A$ such that $A'$ and $T(A')$ are disjoint.
\end{lem}

\begin{thm}\label{thm:cara highly faithful pmp}Let $\Gamma$ be a countable group and fix a measure-preserving ergodic $\Gamma$-action on $(X,\mu)$. Then the following are equivalent:
\begin{enumerate}[(1)]
\item the $\Gamma$-action is highly-faithful;
\item for all finite $F\subseteq \Gamma\setminus\{1\}$, the set $\{x\in X: \forall f\in F, fx\neq x\}$ has positive measure;
\item for all finite $F\subseteq \Gamma$, there exists a positive measure set $A\subseteq X$ such that $(fA)_{f\in F}$ is disjoint;
\item there exists an increasing exhaustive family $(F_n)$ of finite subsets of $\Gamma$ and a sequence of positive measure subsets $(A_n)$ of $X$ such that $(fA_n)_{f\in F_n, n\in\N}$ is disjoint;
\item\label{cond:nice disjoint sets for all Fn} whenever $(F_n)$ is an increasing exhaustive family of finite sets $(F_n)$ of $\Gamma$, there exists a sequence of positive measure subsets $(A_n)$ of $X$ such that $(fA_n)_{f\in F_n, n\in\N}$ is disjoint.
\end{enumerate}
\end{thm}
\begin{proof}
The chain of implications $(5)\impl(4)\impl(3)\impl(2)$ is straightforward. Note that by ergodicity given a countable family of Borel sets, all its members are of positive measure if and only if almost every $\Gamma$-orbit intersects each of its members. In particular condition (2) is satisfied if and only if the $\Gamma$-action onto almost every orbit is highly faithful, so the equivalence $(1)\Leftrightarrow(2)$ holds. 
 Also (5) follows from (4) since given any two exhaustive increasing sequences $(F_n)$, $(F'_n)$ of subsets of $\Gamma$, there exists an increasing map $\varphi:\N\to\N$ such that for all $n\in\N$ we have  $F'_n\subseteq F_{\varphi(n)}$. 

Let us show that (2) implies (3). Let $F$ be a finite subset of $\Gamma$, consider the set $F':=\{f_2\inv f_1: f_1\in F, f_2\in F, f_1\neq f_2\}$. By (2) and an inductive application of Lemma 
\ref{lem: disjoint image}, we find $A\subseteq X$ of positive measure such that for all $f\in F'$, $fA\cap A=\emptyset$. But then for all $f_1\neq f_2\in F$, we have $f_2\inv f_1A\cap A=\emptyset$, so $f_1A\cap f_2A=\emptyset$, which establishes (3).

We now only have to prove that (3) implies (4), so let us assume (3). We fix an increasing exhausting sequence $(F_n)$ of finite subsets of $\Gamma$ such that $1\in F_0$. Using (3) repeatedly, we obtain a sequence $(B_n)$ of positive measure subsets of $X$ such that for all $n\in\N$, the family $(fB_n)_{f\in F_n}$ is disjoint. By inductively taking smaller subsets, we may assume that for all $n\in\N$, 
$$\abs{F_n}\abs{F_{n+1}}\mu(B_{n+1})<\frac 14\mu(B_n).$$
This implies that for all $n\geq 0$ and $m\geq 1$, we have the inequality 
\begin{equation*}
\abs{F_n}\abs{F_{n+m}}\mu(B_{n+m})<\frac 1{4^m}\mu(B_n).\end{equation*}
For all $n\in\N$, let $A_n:=B_n\setminus\bigcup_{m\geq 1} F_{n}\inv F_{n+m}B_{n+m}$. Since $\sum_{m\geq 1} \frac 1{4^m}<1$, the previous inequality  implies that each $A_n$ has positive measure. 

Let $n\geq 0$, $m\geq 1$, $f_1\in F_n$ and $f_2\in F_{n+m}$.
By construction, the set $A_n$ is disjoint from $f_1\inv f_2 B_{n+m}$. Since $B_{n+m}$ contains $A_{n+m}$, we deduce that $A_n$ is disjoint from $f_1\inv f_2 A_{n+m}$ so that $f_1A_n$ is disjoint from $f_2 A_{n+m}$. Since $A_n$ is a subset of $B_n$ whose $F_n$-translates are disjoint, this means that the sequence $(fA_n)_{f\in F_n, n\in\N}$ is made of pairwise disjoint sets as required.
\end{proof}
\begin{rmq}The non-ergodic version of the previous theorem is obtained by asking in (2), (3), (4) and (5) that the sets which are considered intersect almost every orbit. 
\end{rmq}

\section{Residually finite groups and high faithfulness}\label{sec: rf groups and hf}

Let us first recast one definition of residual finiteness in terms of sequences of actions on finite sets.

\begin{df}
Let $\Gamma$ be a countable group, and let $(X_n,\alpha_n,o_n)$ be a sequence of pointed $\Gamma$-actions on finite sets. The sequence is \textbf{asymptotically free} if for all $\gamma\in\Gamma\setminus\{1\}$, there exists $N\in\N$ such that for all $n\geq N$, one has $\gamma\cdot o_n\neq o_n$.
\end{df}

The following lemma is proven exactly as Lemma \ref{lem:highly faithful is locally free}.
\begin{lem}\label{lem:asfreemoves}
Let $\Gamma$ be a countable group, and let $(X_n,\alpha_n,o_n)$ be an asymptotically free sequence of pointed $\Gamma$-actions on finite sets. Then for all finite $F\subseteq \Gamma$, there exists $N\in\N$ such that for all $n\geq N$ and all distinct $\gamma,\gamma'\in F$, one has 
$\gamma o_n\neq \gamma' o_n$.
\end{lem}

\begin{df}
A countable group $\Gamma$ is \textbf{residually finite} if it admits an asymptotically free sequence of pointed actions on finite sets.
\end{df}


The following lemma is well-known and can be used to show that every residually finite group embeds into the full group of any ergodic measure-preserving equivalence relation (see. \cite[4.(E)]{MR2583950} for more on this). We include a proof for completeness.

\begin{lem}\label{lem: S(K) in [R]}Let $\mathcal R$ be a countable measure-preserving ergodic equivalence relation on $(X,\mu)$. Suppose that $K$ is a finite set acted upon by a countable group $\Lambda$, and let $(C_k)_{k\in K}$ be a family of disjoint subsets of $X$, all of the same measure. 

Then there is a homomorphism $\iota:\Lambda\to[\mathcal R_{\restriction\bigcup_{k\in K}C_k}]$ such that for all $x\in\bigcup_{k\in K}C_k$, the $\Lambda$-action on the $\Lambda$-orbit of $x$ is conjugate to the $\Lambda$-action on $K$, and moreover for all $\lambda\in\Lambda$ and all $k\in K$, one has $\iota(\lambda)(C_k)=C_{\lambda(k)}$.  
\end{lem}
\begin{proof}
Let $n$ be the cardinality of the set $K$, then we can suppose that $K=\Z/n\Z$. Since $\mathcal R$ is ergodic and all the $C_k$'s have the same measure, by Proposition \ref{isopar} there is a pre-$n$-cycle $\Phi=\{\varphi_1,...,\varphi_{n-1}\}$ such that for all $i\in\{1,...,n-1\}$, we have $\varphi_i(C_{i-1})=C_i$.   Let $T_\Phi\in[\mathcal R]$ be the associated $n$-cycle. Given $\lambda\in\Lambda$, we then define $\iota(\lambda)$ by:
$$\iota(\lambda)(x)=T_\Phi^{\lambda\cdot k-k}(x) \text{ where }k\in K\text{ is such that }x\in C_k.$$
Note that $\iota(\lambda)$ is well defined because $T^n=\mathrm{id}_X$. It is then straightforward to check that $\iota$ is a homomorphism satisfying the required assumptions.
\end{proof}

\begin{thm}\label{thm:highly faithful rf}
Let $\Gamma$ be a countable group. Consider a measure-preserving highly faithful ergodic $\Gamma$-action on $(X,\mu)$ and let $\Lambda$ be a residually finite countable group. Let $(F_n)$ be an increasing exhaustive family of finite subsets $(F_n)$ of $\Gamma$ such that $1\in F_0$, let $(A_n)$ be a sequence of positive measure subsets $(A_n)$ of $X$ such that $(fA_n)_{f\in F_n, n\in\N}$ is disjoint\footnote{Such a sequence exists by Theorem \ref{thm:cara highly faithful pmp}.}. Fix an asymptotically free sequence $(X_m,\alpha_m,o_m)_{m\in\N}$ of $\Lambda$-actions on finite pointed sets.

Then there exists a measure-preserving  $\Lambda$-action on $(X,\mu)$ which preserves the $\Gamma$-orbits such that the following assertions are true.

\begin{enumerate}[(1)]
\item The induced $\Gamma*\Lambda$-action is highly faithful.
\item The $\Lambda$-action is supported on $\bigsqcup_{f\in F_n,n\in\N} fA_n$ and has only finite orbits.
\item For all $x\in X$, either $x$ is fixed by $\Lambda$ or there exists $n\in\N$ such that the $\Lambda$-action on the $\Lambda$-orbit of $x$  is conjugate to $\alpha_n$.
\item Any $\Lambda$-action which coincides with this action when restricted to $\bigsqcup_{f\in F_n,n\in\N} fA_n$ will induce a highly faithful $\Gamma*\Lambda$-action.
\end{enumerate}
\end{thm}
\begin{proof}
Let $(G_n)_{n\in\N}$ be an increasing exhaustive sequence of finite subsets of $\Lambda$ such that $1\in G_0$. For all $n\in\N$,  let $G'_n:= G_n\setminus\{1\}$ and $F'_n=F_n\setminus\{1\}$.
For all $n\in\N$ and $k\in\{0,...,n\}$, define the following finite subsets of $\Gamma*\Lambda$:
$$\left\lbrace \begin{array}{rl}I_{k,n}&:=\underbrace{(G'_nF'_n)(G'_nF'_n)\cdots (G'_n F'_n)}_{k \text{ times}} G_n \text{ and}\\ J_{k,n}&:=F'_nI_{k,n},\end{array}\right.$$
where by convention $I_{0,n}=G_n$.

Then let $H_n=\bigcup_{k=0}^n(I_{k,n}\cup J_{k,n})$. The sequence $(H_n)$ is clearly an increasing exhaustive sequence of finite subsets of $\Gamma*\Lambda$.

We will define the $\Lambda$-action piece by piece, so that for every $n\in\N$, the set $\bigcup_{f\in F_n} fA_n$ is $\Lambda$-invariant and $A_n$ witnesses the fact that the $\Gamma*\Lambda$-action is highly faithful for the finite set $H_n$ in the following sense: there is a smaller $A'_n\subseteq A_n$ such that the collection $(hA'_n)_{h\in H_n}$ is made of disjoint sets.

So let us fix $n\in\N$. 
Since the sequence of pointed $\Lambda$-actions $(X_m,\alpha_m,o_m)_{m\in\N}$ is asymptotically free, by Lemma \ref{lem:asfreemoves} we find $m\in\N$ such that for all distinct $\lambda,\lambda'\in I_{0,n}$ we have $\alpha_m(\lambda)(o_m)\neq \alpha_m(\lambda')(o_m)$. Let $k_m$ be the cardinality of the set $X_m$; we may as well assume that $X_m=\{0,...,k_m-1\}$ and $o_m=0$. We also fix a subset $A'_n$ of $A_n$ of measure $\epsilon_n$, where $\epsilon_n$ is a fixed positive real such that $\epsilon_n<\frac{\mu(A_n)}{k_m\abs{H_n}}$.

Since $\epsilon_n<\frac{\mu(A_n)}{k_m+1}$ we can find disjoint subsets $C_1,...,C_{k_m-1}\subseteq A_n$ of measure $\epsilon_n$ which are all disjoint from $A'_n$, and we let $C_0:=A'_n$.
 
We define the $\Lambda$-action on $\bigcup_{k=0}^{k_m-1} C_k$ via a chosen homomorphism $\iota:\Lambda\to[\mathcal R_{\restriction{\bigcup_{k\in X_m}C_k}}]$ provided by Lemma \ref{lem: S(K) in [R]} applied to the $\Lambda$-action $\alpha_m$ on $X_m$, where $\mathcal R$ is the measure-preserving equivalence relation induced by the $\Gamma$-action on $(X,\mu)$. Recall that $A'_n=C_0$. Since $\lambda(C_0)=C_{\alpha_m(\lambda)(0)}$, we see that by construction the family $(\lambda(A'_n))_{\lambda\in I_{0,n}}$ is made of disjoint subsets of $A_n$. Moreover, since  the family $(f(A_n))_{f\in F_n}$ is made of disjoint sets, we see that the equality $J_{0,n}=F'_nI_{0,n}$ yields that the family
$(h(A'_n))_{h\in I_{0,n}\cup J_{0,n}}$
is made of disjoint sets.

The above setup initializes the following construction for $l=0$: inductively on $l\in\{0,...,n\}$ we will now define the $\Lambda$-action on bigger and bigger sets. So suppose that for some $l\in\{0,...,n-1\}$, we have constructed a $\Lambda$-action on $$\bigsqcup_{h\in \bigcup_{k=0}^lI_{k,n}}h(A'_n)\sqcup \bigsqcup_{h\in \bigcup_{k=0}^{l-1} J_{k,n}}h(A'_n)$$ satisfying the following assumptions
\begin{enumerate}[(a)]
\item[(a)] the $\Lambda$-action is conjugate to $\alpha_m$ when restricted to any orbit,
\item[(b)] the $\Lambda$-action preserves the $\Gamma$-orbits and
\item[(c$_l$)] for any $k\leq l$ and any $h\in I_{k,n}$, the set $h(A'_n)$ is a subset of $A_n$.
\end{enumerate}

Since  the family $(f(A_n))_{f\in F_n}$ is made of disjoint sets, condition (c) and the fact that $(h(A'_n))_{h\in \bigcup_{k=0}^lI_{k,n}}$ is disjoint implies that the family
$$(h(A'_n))_{h\in \bigcup_{k=0}^l J_{k,n}}$$
is actually made of disjoint sets which are all disjoint from $A_n$. Let us fix a family $$(C_{h,k})_{{k\in\{1,...,k_m-1\}},h\in J_{l,n}}$$ of disjoint subsets of $A_n$ of measure $\epsilon_n$ such that they are also disjoint from the set $\bigsqcup_{h\in \bigcup_{k=0}^lI_{k,n}}h(A'_n)$ (here we fully use the condition $\epsilon_n<\frac{\mu(A_n)}{k_m\abs{H_n}}$). For $h\in J_{l,n}$, let $C_{h,0}:=h(A'_n)$. Using again Lemma \ref{lem: S(K) in [R]}, we define for every $h\in J_{l,n}$ the $\Lambda$-action on $\bigsqcup_{i=0}^{k_m-1}C_{h,i}$ so that it preserves the $\mathcal R$-classes, that it is conjugate to $\alpha_m$ when restricted to an orbit and that for all $i\in\{0,...,k_m\}$, we have $\lambda(C_{h,i})=C_{h,\alpha_m(\lambda)(i)}$.

Now  the $\Lambda$-action is also defined on $\bigsqcup_{h\in I_{l,n}}\bigsqcup_{k=0}^{k_m-1}C_{h,k}$. And every $h\in I_{l+1,n}$ is of the form $h=\lambda \bar h$ for some $\lambda\in G'_n$ and $\bar h\in J_{l,n}$ so that $h(A'_n)=\lambda \bar h(A'_n)=\lambda(C_{\bar h,0})=C_{\bar h,\alpha_m(\lambda)(0)}$. Since $\alpha_m(\lambda)(0)\neq 0$, we see that the family $(h(A'_n))_{h\in \bigcup_{k=0}^{l+1}I_{k,n}}$ is made of disjoint subsets of $A_n$. This implies that the family $(h(A'_n))_{h\in \bigcup_{k=0}^{l+1}I_{k,n}\cup J_{k,n}}$ is disjoint, so we have now constructed 
a $\Lambda$-action on $$\bigsqcup_{h\in \bigcup_{k=0}^{l+1}I_{k,n}}h(A'_n)\sqcup \bigsqcup_{h\in \bigcup_{k=0}^{l} J_{k,n}}h(A'_n)$$
which satisfies conditions (a), (b) and (c$_{l+1}$). 

 Now for every $n\in\N$ we have defined the $\Lambda$-action on a subset of $\bigcup_{f\in F_n}fA_n$ and we declare it to be trivial anywhere else. By construction, for every $n\in\N$ the family $(h(A'_n))_{h\in H_n}$ is made of disjoint sets so that the induced $\Gamma*\Lambda$-action is highly faithful: condition (1) is thus satisfied. Conditions (2), (3) and (4) also follow from the construction.
\end{proof}

\section{Topological generators in the hyperfinite case}\label{sec:top gen hyperfinite}

In this section, we get more flexibility in the construction from \cite{lm14nonerg} of topological generators for the full group of the hyperfinite ergodic measure-preserving equivalence relation $\mathcal R_0$. 

\subsection{The equivalence relation $\mathcal R_0$}

Recall that $\mathcal R_0$ is defined on the space of infinite binary sequences $\{0,1\}^\N$ equipped with the product Bernoulli probability measure $\bigotimes_{n\in\N}\frac 12(\delta_0+\delta_1)$. By definition, two sequences $(x_i)_{i\in\N}$ and $(y_i)_{i\in\N}$ are $\mathcal R_0$-equivalent if they are the same up to a finite number of indices, that is, if there is $N\in\N$ such that for all $i\geq N$, we have $x_i=y_i$.

Let us now introduce a bit of notation. Any finite binary sequence $s\in \{0,1\}^n$ defines a subset $N_s$ of the product space $\{0,1\}^\N$ consisting of all the sequences starting by $s$, i.e. $$N_s:=\{x\in \{0,1\}^\N: x_i=s_i\text{ for }  i\in\{0,..., n-1\}\}.$$
We can see elements $a\in \{0,1\}^n$ and $b\in\{0,1\}^\N\cup\bigcup_{n\in\N}\{0,1\}^n$ as words in $\{0,1\}$, and denote their concatenation by $a\smallfrown b$. For $\epsilon\in\{0,1\}$ and $n\in\N$,  $\epsilon^n$ is the word $(x_i)_{i=1}^n\in \{0,1\}^n$ defined by $x_i=\epsilon$.

Let $n\in\N$. The group $\mathfrak S_{\{0,1\}^n}$ is the group of permutations of the set $\{0,1\}^n$. There is a natural inclusion $\alpha_n:\mathfrak S_{\{0,1\}^n}\into\mathfrak S_{\{0,1\}^{n+1}}$ given by 
$$\alpha_n(\sigma)(x_0,...,x_{n})=(\sigma(x_0,...,x_{n-1}),x_{n})$$
for $\sigma\in\mathfrak S_{\{0,1\}^n}$ and $(x_0,...,x_{n})\in\{0,1\}^{n+1}$. Let $\mathfrak S_{\{0,1\}^{<\infty}}$ be the inductive limit of these groups, called the group of \textbf{dyadic permutations}.

The key feature of $\mathfrak S_{\{0,1\}^{<\infty}}$ is that it acts in a measure-preserving way on $(\{0,1\}^\N,\lambda)$ as follows: for $\sigma\in\mathfrak S_{\{0,1\}^n}$, $s\in \{0,1\}^n$ and $x\in \{0,1\}^\N$,
$$\sigma(s\smallfrown x)=\sigma(s)\smallfrown x.$$
It is straightforward to check that the orbit equivalence relation induced by this action is $\mathcal R_0$. To avoid confusion, when we see $\mathfrak S_{\{0,1\}^n}$ as a subgroup of $[\mathcal R_0]$ we denote it by $\Stild_{\{0,1\}^n}$.

The following proposition belongs to the folklore, for a proof see \cite[Prop. 3.8]{MR2583950}.
\begin{prop}\label{permdense}The group of dyadic permutations is dense in the full group of $\mathcal R_0$.
\end{prop}

The \textbf{odometer} is the map $T_0\in\Aut(\{0,1\}^\N,\lambda)$ defined by
$$(x_i)_{i\in\N}\in \{0,1\}^\N\mapsto 0^{n-1}1\smallfrown (x_i)_{i>n},$$
where $n$ is the first integer such that $x_n=0$ (note that this is well defined on a set of full measure). This can be understood as adding $(1,0,0,...)$ to $(x_i)_{i\in\N}$ with right carry. One can check that $T_0$ generates $\mathcal R_0$.

Let $n\in\N$, then we define a finite odometer $\sigma_n\in\mathfrak S_{\{0,1\}^n}$ by 
$$\sigma_n((s_i)_{i=0}^{n-1})=\left\{\begin{array}{ll}0^n &\text{if }(s_i)=1^n  \\0^{k-1}1\smallfrown (s_i)_{i>k} & \text{else, where }k \text{ is the first integer such that }s_k=0. \end{array}\right.
$$
We denote by $T_n$ the corresponding element in $\Stild_{\{0,1\}^n}$.
Note that by definition, $T_n$ and $T_0$ coincide on $\{0,1\}^\N\setminus N_{1^n}$. 

\subsection{Modified topological generators in the hyperfinite case}

Let $n\geq 2$, and define $\tau_n\in\mathfrak S_{\{0,1\}^n}$ to be the transposition which exchanges $0^{n-1}1$ and $1^{n-1}0$. Let $U_n$ be the corresponding element of $\Stild_{\{0,1\}^n}$, that is, the element of $[\mathcal R_0]$ implementing the action of $\tau_n$ on $2^\N$. Note that the support of $U_n$ is $N_{0^{n-1}1}\sqcup N_{1^{n-1}0}$, so that the supports of the $U_n$'s are all disjoint. 

The next lemma boils down to the well-known fact that the symmetric group over $2^n$ elements is generated by any $2^n$-cycle $\sigma$ along with a transposition $\tau$ exchanging two $\sigma$-consecutive elements. For a detailed proof see \cite[Lem. 4.3]{lm14nonerg}.

\begin{lem}\label{generiota}The group $\Stild_{\{0,1\}^n}$ is contained in the group generated by $T_0$ and $U_n$. \end{lem}

We see that if we could produce $U\in[\mathcal R_0]$ such that the closed subgroup generated by $U$ contains infinitely many $U_n$'s, the fact that $\Stild_{\{0,1\}^{<\infty}}$ is dense $[\mathcal R_0]$ coupled with the previous lemma would yield that $T_0$ and $U$ generate a dense subgroup of $[\mathcal R_0]$. Although this cannot be done, the main idea of \cite{lm14nonerg} is to find $U\in[\mathcal R_0]$ such that the closed subgroup generated by $U$ contains infinitely many $U_n$'s up to an error which tends very fast to zero, so that $\{T_0,U\}$  generates a dense subgroup of $[\mathcal R_0]$. 

To this end, we now fix for every $n\in\N$ a constant $\kappa(n)$ such that any element of $\Stild_{\{0,1\}^n}$ can be written as a word in $U_n$ and $T_0$ of length less than $\kappa(n)$. For all $p,q\in\N$, we will use the function
$$\sqrt[2^p]{\cdot}:\Stild(\{0,1\}^q)\to\Stild(\{0,1\}^{p+q})$$
defined in \cite{lm14nonerg}, which satisfies that for all $U\in\Stild(\{0,1\}^{<\infty})$, one has $\left(\sqrt[2^p]{U}\right)^{2^p}=U$ and $\sqrt[2^p]{U}$ has the same support as $U$.

If $T\in\Aut(X,\mu)$ and $A$ is a Borel subset of $X$ which is $T$-invariant, we define the induced transformation $T_A$ with respect to $A$ by: for all $x\in X$,
$$T_A(x)=\left\{\begin{array}{ll}T(x) & \text{ if }x\in A\\x & \text{ else.}\end{array}\right.$$
 We can now state and prove a version of \cite[Thm. 1.4]{lm14nonerg}  where we allow for some error. The argument is very close to the original one, but we give a full proof for the convenience of the reader.

\begin{thm}\label{thm:topogenR0}Given any $\epsilon>0$, there exists an increasing sequence of integers $(n_k)_{k\in\N}$ and a sequence of positive reals $(\delta_k)_{k\in\N}$ such that whenever we have for all $k\in\N$ a $\sqrt[2^{k-1}]{U_{n_k}}$-invariant set $B_k\subseteq \supp U_{n_k}$ with $\mu(B_k)>\mu(\supp U_{n_k})-\delta_k$,
if we let 
$$U:=\prod_{k=0}^{+\infty}\sqrt[2^{k}]{U_{n_k}}_{B_k},$$
then the set $\{T_0,U\}$ generates a dense subgroup of $[\mathcal R_0]$ and  we have $\mu(\supp U)<\epsilon$.
\end{thm}
\begin{proof}
Fix a sequence $(\epsilon_k)$ of positive real numbers such that $\epsilon_k\to 0$. 

\begin{claim}It suffices find sequences $(n_k)$ and $(\delta_k)$ with $\sum_{k\in\N}2^{-n_k}<\epsilon$ such that whenever we have for all $k\in\N$ a $\sqrt[2^{k-1}]{U_{n_k}}$-invariant set $B_k\subseteq \supp U_{n_k}$ with $\mu(B_k)>\mu(\supp U_{n_k})-\delta_k$,
if we let
$$U:=\prod_{k=0}^{+\infty}\sqrt[2^{k-1}]{U_{n_k}}_{B_k},$$
then for all $k\in\N$, there exists  $U'\in \la U\ra$ such that $d_u(U_{n_k},U')<\epsilon_k/\kappa(n_k)$. 
\end{claim}
\begin{proof}[Proof of the claim]
Assuming that the above conditions are satisfied, fix $k\in\N$ and $U'\in\la U\ra$ such that $d_u(U_{n_k},U')<\epsilon_k/\kappa(n_k)$. Since every element of $\Stild_{\{0,1\}^{n_k}}$ can be written as a word in $T_0$ and $U_{n_k}$ of length less than $\kappa(n_k)$ (see Lem. \ref{generiota} and the definition of $\kappa(n_k)$), we deduce that every element of $\Stild_{\{0,1\}^{n_k}}$ belongs to $\la T_0,U\ra$ up to an error less than $\epsilon_k$. Now $\epsilon_k\to 0$ so the closed group generated by $\{T_0,U\}$ contains $\Stild_{\{0,1\}^{<\infty}}$, hence $\overline{\la T_0,U\ra}=[\mathcal R_0]$ by Proposition \ref{permdense}. And since for all $k\in\N$ we have $\mu(\supp \sqrt[2^{k}]{U_{n_k}}_{B_k})\leq\mu(\sqrt[2^{k}]{U_{n_k}})=2^{-n_k}$, we deduce that $\mu(\supp U)\leq\sum_{k\in\N}2^{-n_k}<\epsilon$ as desired.
\end{proof}

We now build by induction an increasing sequence $(n_k)$ such that $\sum_{k\in\N}2^{-n_k}<\epsilon$ and for all $k\in\N$, we have \begin{equation}\label{nk small}2^{-n_{k+1}-2}<\frac{\epsilon_k}{2\kappa(n_k)} .\end{equation} Then we choose for every $k\in\N$ a positive $\delta_k$ such that \begin{equation}\label{deltak small}\delta_k<\frac{\epsilon_k}{2\kappa(n_k)}\end{equation}
Let us show that such sequences $(n_k)$ and $(\delta_k)$ satisfy the hypotheses of the claim, which will end the proof. So suppose that for every $k\in\N$ we have a $\sqrt[2^{k-1}]{U_{n_k}}$-invariant set $B_k\subseteq \supp U_{n_k}$ with $\mu(B_k)>\mu(\supp U_{n_k})-\delta_k$. First note that all the $\sqrt[2^{k}]{U_{n_k}}_{B_k}$ have disjoint supports, so they commute. We fix $k\in\N$ and compute
\begin{align*}
U^{2^{k}}&=\prod_{l=0}^{+\infty}{\left({\sqrt[2^{l}] {U_{n_l}}}_{B_l}\right)}^{2^{k}}\\
&=\prod_{l=0}^{k-1} \left({U_{n_l}}_{B_l}\right)^{2^{k-l}}\cdot {U_{n_k}}_{B_k}\cdot\prod_{l=k+1}^{+\infty}{\left({\sqrt[2^{l}] {U_{n_l}}}_{B_l}\right)}^{2^{k}}
\end{align*}
Because the $U_{n_l}$'s are involution, the first product is equal to the identity, so that
\begin{equation}\label{product}
U^{2^{k}}={U_{n_k}}_{B_k}\cdot\prod_{l=k+1}^{+\infty}{\left({\sqrt[2^{l}] {U_{n_l}}}_{B_l}\right)}^{2^{k}}.
\end{equation}
We now check that the error term $W_k:=\prod_{l=k+1}^{+\infty}{\left({\sqrt[2^{l}] {U_{n_l}}}_{B_l}\right)}^{2^{k}}$ is small. Because for every $l\in\N$,  ${\left({\sqrt[2^{l}] {U_{n_l}}}_{B_l}\right)}^{2^{k}}$ has same support as ${U_{n_l}}_{B_l}$, the support of $W_k$ has measure smaller than 
\begin{equation}\label{sum}\sum_{l=k+1}^{+\infty}\lambda(\supp U_{n_l})\leq\sum_{l=k+1}^{+\infty}\frac 1{2^{n_l-1}}\end{equation}
Since $(n_l)_{l\in\N}$ is increasing, we have for all $l\geq k+1$, $$\frac 1{2^{n_l-1}}\leq \frac 1{2^{n_{k+1}+(l-k-2)}}$$
We can now bound the right-hand term in (\ref{sum}) and get the inequality 
\begin{align*}\mu(\supp W_k)&\leq\frac1{2^{n_{k+1}}}\cdot \sum_{l={k+1}}^{+\infty} \frac 1{2^{l-k-2}}\\
&\leq\frac 4{2^{n_{k+1}}}\\
&\leq \frac{\epsilon_k}{2\kappa(n_k)},
\end{align*}
the latter inequality being a direct consequence of (\ref{nk small}).
From this and equation (\ref{product}) we deduce
\begin{align*}
d_u(U^{2^{k}},{U_{n_k}}_{B_k})&\leq \mu(\supp(W_k))\\
&\leq\frac{\epsilon_k}{2\kappa(n_k)}.
\end{align*}
Since $B_k$ is a $U_{n_k}$-invariant subset of the support of $U_{n_k}$ such that $\mu(B_k)>\mu(\supp U_{n_k})-\delta_k$ , we have $d_u({U_{n_k}}_{B_k},U_{n_k})<\delta_k<\frac{\epsilon_k}{2\kappa(n_k)}$ by (\ref{deltak small}). We deduce that 
$$d_u(U^{2^k},U_{n_k})<\frac{\epsilon_k}{\kappa(n_k)}$$
so that the theorem now follows from the claim.
\end{proof}

\section{Proof of the main theorem}\label{sec:pf of main thm}

\subsection{A lemma on commuting elements}

\begin{lem}
Let $T,U\in \Aut(X,\mu)$ have disjoint supports, and suppose that there are two relatively prime numbers $p, q\geq 2$ such that \begin{itemize}
\item every $T$-orbit is finite and its cardinality divides a power of $p$ and
\item every $U$-orbit is finite and its cardinality divides a power of $q$.
\end{itemize}
 Then both $T$ and $U$ belong to the closure of the group generated by $TU$ for the uniform topology.
\end{lem}
\begin{proof}
Since $U=T\inv (TU)$, it suffices to show that $T$ belongs to the closure of $\la TU\ra$.  Let $\epsilon>0$, find $N\in\N$ large enough so that there is a Borel set $A$ such that $\mu(A)>1-\epsilon$ and for all $x\in A$;
$$ \abs{\Orb_T(x)}\leq p^N\text{ and }\abs{\Orb_U(x)}\leq q^N$$
Since $p^N$ and $q^N$ are relatively prime, there is $l\in\N$ such that $lq^N\equiv1$ mod $p^N$. 

Note that $T$ and $U$ commute since they have disjoint support and fix $x\in A$. If $x$ belongs to the support of $T$ then $(TU)^{lq^{N}}(x)=T^{lq^N}(x)=T(x)$ because the cardinality of the $T$-orbit of $x$ divides a power of $p$ no greater than $p^N$ and $lq^N\equiv1$ mod $p^N$. If $x$  belongs to the support of $U$ then $(TU)^{lq^{N}}(x)=U^{lq^{N}}(x)=x$ because the cardinality of the $U$-orbit of $x$ divides a power of $q$ no greater than $q^N$, so $(TU)^{lq^{N}}(x)=T(x)$. And if $x$ neither belongs to the support of $T$ nor to the support of $U$, then $(TU)^{lq^{N}}(x)=x=T(x)$. So for all $x\in A$, we have $(TU)^{lq^{N}}(x)=T(x)$. As $\mu(A)>1-\epsilon$, we deduce that $d_u((TU)^{lq^{N}},T)<\epsilon$ and we conclude that $T$ belongs to the closure of $\la TU\ra$.
\end{proof}

A proof by induction yields the following useful corollary.

\begin{cor}\label{cor:obtain elts with disjt sup}
Let $T_1,..,T_n\in \Aut(X,\mu)$ have disjoint supports, and suppose that there are $n$ pairwise relatively prime numbers $p_1,...,p_n\geq 2$ such that for every $k\in\{1,...,n\}$, every $T_k$-orbit is finite and its cardinality divides a power of $p_k$.

Then for all $k\in\{1,...,n\}$, $T_k$ belongs to the closure of the group generated by the product $T_1T_2\cdots T_n$ for the uniform topology.
\end{cor}

\subsection{Proof of Theorem \ref{thm:main thm}}

Let us start with an ergodic equivalence relation $\mathcal R$ such that $\mathrm{Cost(\mathcal R)}<m+1$ for some $m\in\N$. Our goal is to find $m+1$ topological generators for the full group of $\mathcal R$ so that the induced $\mathbb F_{m+1}$-action is highly faithful and amenable onto almost every orbit.

By \cite[Thm. 4]{MR0131516}\footnote{See also \cite[9.3.2]{MR776417} for a statement and a proof with a less operator-algebraic flavour.}, we may and do assume that $X=\{0,1\}^\N$ equipped with the product Bernoulli probability measure $\mu=\bigotimes_{n\in\N}\frac 12(\delta_0+\delta_1)$, and that the odometer $T_0$ belongs to the full group of $\mathcal R$. 

Lemma III.5 in \cite{MR1728876} provides a graphing $\Phi$ such that $\Cost(\Phi)<m$ and $\{T_0\}\cup\Phi$ generates $\mathcal R$. Let 
$$c=\frac{\Cost(\Phi)}{m}<1,$$
and fix some odd $p\in\N$ such that $(\frac {p+2} p)c<1$. Splitting the domains of the partial automorphisms in $\Phi$, we find $\Phi_1$,...,$\Phi_m$ of cost $c$ such that $\Phi=\Phi_1\cup\cdots\cup \Phi_m$.

The map $T_0$ induces a free (in particular highly faithful) ergodic $\Z$-action, so we apply item (5) of Theorem \ref{thm:cara highly faithful pmp} and find a sequence $(A_n)_{n\in\N}$ of non-null Borel subsets of $X$ such that the family $(T_0^{i}(A_n))_{\abs i\leq n, n\in\N}$ is disjoint. Up to taking smaller non-null subsets $A'_n\subseteq A_n$ for all $n\in\N$, we can assume that $\mu(\bigsqcup_{n\in\N}\bigsqcup_{i=-n}^nT_0^{i}(A_n))<\epsilon$.

Let $\epsilon:=\frac{1-(\frac {p+2} p)c}2$, using Theorem \ref{thm:topogenR0} we fix  an increasing sequence of integers $(n_k)_{k\in\N}$ and a sequence of positive reals $(\delta_k)_{k\in\N}$ such that whenever we have for all $k\in\N$ a $\sqrt[2^{k-1}]{U_{n_k}}$-invariant set $B_k\subseteq \supp U_{n_k}$ with $\mu(B_k)>\mu(\supp U_{n_k})-\delta_k$,
if we let 
$$U:=\prod_{k=0}^{+\infty}\sqrt[2^{k}]{U_{n_k}}_{B_k},$$
then the set $\{T_0,U\}$ generates a dense subgroup of $[\mathcal R_0]$ and  we have $\mu(\supp U)<\epsilon$, where $\mathcal R_0$ is the measure-preserving equivalence generated by $T_0$.

\begin{claim}We can also assume that for all $k\in\N$,
$$\mu\left(\left(\bigsqcup_{n\in\N}\bigsqcup_{i=-n}^nT_0^i(A_n)\right)\cap \supp U_{n_k}\right)<\frac{\delta_k}{2^{k+1}}.$$
\end{claim}
\begin{proof}
Let $n\in\N$. Find $l\in\N$ such that the set $\bigcup_{k\geq l}\bigcup_{i=-n}^nT_0^i(\supp U_{n_k})$ has measure less than $\mu(A_n)/2$, and set $A'_n:= A_n\setminus \bigcup_{k\geq l}\bigcup_{i=-n}^nT_0^i(\supp U_{n_k})$ which is non-null. Then by construction for all $k\geq l$, the set $\bigsqcup_{i=-n}^nT_0^i(A'_n)$ is disjoint from $\supp U_{n_k}$. 

We now take $A''_n\subseteq A'_n$ non-null such that for all $k<l$,
$$\mu\left(\bigsqcup_{i=-n}^nT_0^i(A''_n)\right)< \frac{\delta_k}{2^{k+n+2}}$$
Then for all $k,n\in\N$, we have
$$\mu\left(\bigsqcup_{i=-n}^nT_0^i(A''_n)\cap\supp U_{n_k}\right)< \frac{\delta_k}{2^{k+n+2}}$$
But now a straightforward calculation yields that for all $k\in\N$
$$\mu\left(\left(\bigsqcup_{n\in\N}\bigsqcup_{i=-n}^nT_0^i(A''_n)\right)\cap \supp{U_{n_k}} \right)< \frac{\delta_k}{2^{k}}$$
so the sequence $(A''_n)$ is as desired.
\end{proof}

For all $k\in\N$, let 
$$\overline B_k:=\bigcup_{j=0}^{2^{k+1}-1}\sqrt[2^{k}]{U_{n_k}}^{\,j}\left(\supp U_{n_k}\cap \left(\bigsqcup_{n\in\N}\bigsqcup_{i=-n}^nT_0^i(A_n)\right)\right)$$
 Since $\sqrt[2^{k}]{U_{n_k}}$ has order $2^{k+1}$, the set $\overline B_k$ is $\sqrt[2^{k}]{U_{n_k}}$-invariant. Moreover by the previous claim $\mu(\overline B_k)<\delta_k$. We let $B_k:=\supp U_{n_k}\setminus \overline B_k$, and then we define
 $$U:=\prod_{k=0}^{+\infty}\sqrt[2^{k}]{U_{n_k}}_{B_k},$$
By construction, the set $\{T_0,U\}$ generates a dense subgroup of $[\mathcal R_0]$ and  we have $\mu(\supp U)<\epsilon$.\\

We now let $B:=\left(\bigsqcup_{n\in\N}\bigsqcup_{i=-n}^nT_0^i(A_n)\right)\cup\supp U$. Note that by construction $\mu(B)<1-\frac {p+2}pc$.
Let $D_1,...,D_{p+2}$ be pairwise disjoint subsets of $X\setminus B$, of measure $\frac c p$ each. For all $i\in\{1,...,m\}$ we use Proposition \ref{isopar} to pre- and post-compose the partial isomorphisms of $\Phi_i$ by elements in $[[\mathcal R_0]]$ so that each $\Phi_i$ becomes a pre-$(p+1)$-cycle $\Phi_i=\{\varphi_1^i,\varphi_2^i,...,\varphi_p^i\}$ where $\varphi^i_j: D_j\to D_{j+1}$ for all $j\in\{1,...,p\}$. Note that this operation preserves the fact that $\mathcal R$ is generated by $\{T_0\}\cup\Phi_1\cup\cdots\cup \Phi_m$. 

Now choose $\psi\in[[\mathcal R_0]]$ with domain $D_{p+1}$ and range $D_{p+2}$, and add it to every $\Phi_i$. We get $m$ pre-$(p+2)$-cycles $\tilde\Phi_i=\Phi_i\cup\{\psi\}$, and $\{T_0\}\cup\tilde\Phi_1\cup\cdots\cup \tilde\Phi_n$ still generates $\mathcal R$. Consider the associated $(p+2)$-cycles $C_{\tilde\Phi_i}$.

\begin{claim}
The $m+2$ elements $T_0,U,C_{\tilde\Phi_1},...,C_{\tilde\Phi_m}$ generate a dense subgroup of the full group of $\mathcal R$.
\end{claim}
\begin{proof}
Let $G$ be the closed group generated by $\{T_0,U,C_{\tilde\Phi_1},...,C_{\tilde\Phi_m}\}$.
Recall that $T_0$ and $U$ have been chosen so that they generate together a dense subgroup of $[\mathcal R_0]$, so $G$ contains $[\mathcal R_0]$. 
 
Because $\psi$ is a partial isomorphism of $\mathcal R_0$, we have $[\mathcal R_{\{\psi\}}]\subseteq[\mathcal R_0]\subseteq G$. 
Since for all $i\in\{1,...,m\}$ we have $\psi\in\tilde\Phi_i$ and $C_{\tilde\Phi_i}\in G$, Proposition \ref{isopgen} implies that $G$ contains $[\mathcal R_{\tilde\Phi_i}]$.
But $\mathcal R$ is the join of $\mathcal R_0$, $\mathcal R_{\tilde\Phi_1},...,\mathcal R_{\tilde\Phi_m}$, so by Theorem \ref{ktdense} we are done.
\end{proof}

For each $n\in\N$, let $A'_n$ and $A''_n$ be two non-null disjoint subsets of $A_n$ such that $A_n=A'_n\sqcup A''_n$. Let $q\in\N$ be an odd prime number which does not divide $p+2$.
By \cite{MR0036767}, the group $\mathbb F_m$ is a residually $q$-finite group, so we can find an asymptotically free sequence of pointed $\mathbb F_m$-actions $(X_n,\alpha_n,o_n)$ such that for all $n\in\N$ and all $\lambda\in\mathbb F_m$, the permutation $\alpha_n(\lambda)$ has order $q^k$ for some $k\in\N$. 

Let $A=\bigsqcup_{n\in\N}\bigsqcup_{i=-n}^nT_0^i(A'_n)$. We now apply Theorem \ref{thm:highly faithful rf} to $\Gamma=\Z$ through the action induced by $T_0$, $\Lambda=\mathbb F_m$, the sequence of actions $(X_n,\alpha_n,o_n)$ and the sequence of sets $(A'_n)$ such that the sequence $(T_0^i(A'_n))_{i=-n}^n$ is made of disjoint sets. We thus obtain a $\mathbb F_m$-action supported on $A$ which preserves the $\mathcal R$-classes and
satisfies the following conditions.
\begin{enumerate}[(1)]
\item The induced $\Z*\mathbb F_m$-action is highly faithful.
\item The $\mathbb F_m$-action is supported on $A$ and has only finite orbits.
\item For all $x\in X$, there exists $n\in\N$ such that the $\mathbb F_m$-action on the $\mathbb F_m$-orbit of $x$  is conjugate to $\alpha_n$. 
\item \label{cond: Fm action same implies hf}Any $\mathbb F_m$-action whose restriction to $A$ coincides with this action will induce a highly faithful $\Z*\mathbb F_m$-action.
\end{enumerate}
The $\mathbb F_m$-action we just obtained is determined by the elements of the full group induced by its standard generators which we denote by $V_1,...,V_m\in[\mathcal R]$.

By our hypothesis on the sequence of actions on finite sets $(\alpha_n)$, for all $i\in\{1,...,m\}$, every $V_i$-orbit has cardinality $q^k$ for some $k\in\N$. Moreover, we have that $U$, $V_1$ and $C_{\tilde\Phi_1}$ have disjoint supports and that for all $i\in\{2,...,m\}$, $V_i$ and $C_{\tilde\Phi_i}$ have disjoint supports.

By corollary \ref{cor:obtain elts with disjt sup}, the elements $U$ and $C_{\tilde\Phi_1}$ belong to the closure of the group generated by $UV_1C_{\tilde\Phi_1}$, and for all $i\in\{2,...,m\}$, $C_{\tilde\Phi_i}$ belongs to the closure of the group generated by $V_iC_{\tilde \Phi_i}$. So by the previous claim the group generated by the $m+1$ elements $$ T_0,UV_1C_{\tilde\Phi_1},V_2C_{\tilde\Phi_2},...,V_mC_{\tilde\Phi_m}$$ is a dense subgroup of the full group of $\mathcal R$. Let us show that the associated $\mathbb F_{m+1}$-action has all the desired properties. 

First, the fact that the $m$ last generators $UV_1C_{\tilde\Phi_1},V_2C_{\tilde\Phi_2},...,V_mC_{\tilde\Phi_m}$ act trivially on $\bigsqcup_{n\in\N}\bigsqcup_{i=-n}^nT_0^i(A''_n)$ implies that for almost all $x\in X$, the restriction of  the Schreier graph of the $\mathbb F_{m+1}$-action  on the orbit of $x$ contains arbitrarily long intervals, so the $\mathbb F_{m+1}$-action is amenable onto almost every orbit.

Then recall that $V_1$,...,$V_m$ induce an $\mathbb F_m$-action which satisfies conditions (1)-(4) above. Moreover, $V_1$ and $UV_1C_{\tilde\Phi_1}$ have the same restriction to $A$ and for all $i\in\{2,...,m\}$, $V_i$ and $V_iC_{\tilde\Phi_i}$ have the same restriction to $A$. By (\ref{cond: Fm action same implies hf}), this implies that that the $\Z*\mathbb F_m=\mathbb F_{m+1}$-action that we have built is highly faithful, which ends the proof of Theorem \ref{thm:main thm}.

\bibliographystyle{alpha}
\bibliography{/Users/francoislemaitre/Dropbox/Maths/biblio}

\begin{thebibliography}{AGV14}

\bibitem[AGV14]{MR3165420}
Mikl{{\'o}}s Ab{{\'e}}rt, Yair Glasner, and B{{\'a}}lint Vir{{\'a}}g.
\newblock Kesten's theorem for invariant random subgroups.
\newblock {\em Duke Math. J.}, 163(3):465--488, 2014.

\bibitem[Bow15]{MR3420547}
Lewis Bowen.
\newblock Invariant random subgroups of the free group.
\newblock {\em Groups Geom. Dyn.}, 9(3):891--916, 2015.

\bibitem[Dye59]{MR0131516}
Henry~A. Dye.
\newblock On groups of measure preserving transformation. {I}.
\newblock {\em Amer. J. Math.}, 81:119--159, 1959.

\bibitem[EG14]{Eisenmann:2014oq}
Amichai Eisenmann and Yair Glasner.
\newblock Generic {IRS} in free groups, after {B}owen.
\newblock {\em to appear in Proc. Amer. Math. Soc.}, 2014.

\bibitem[FMS15]{MR3411128}
Pierre Fima, Soyoung Moon, and Yves Stalder.
\newblock Highly transitive actions of groups acting on trees.
\newblock {\em Proc. Amer. Math. Soc.}, 143(12):5083--5095, 2015.

\bibitem[Gab00]{MR1728876}
Damien Gaboriau.
\newblock Co\^ut des relations d'{\'e}quivalence et des groupes.
\newblock {\em Invent. Math.}, 139(1):41--98, 2000.

\bibitem[Gab10]{MR2827853}
Damien Gaboriau.
\newblock Orbit equivalence and measured group theory.
\newblock In {\em Proceedings of the {I}nternational {C}ongress of
  {M}athematicians. {V}olume {III}}, pages 1501--1527, New Delhi, 2010.
  Hindustan Book Agency.

\bibitem[Gla03]{MR1958753}
Eli Glasner.
\newblock {\em Ergodic theory via joinings}, volume 101 of {\em Mathematical
  Surveys and Monographs}.
\newblock American Mathematical Society, Providence, RI, 2003.

\bibitem[Gro99]{MR1699320}
Misha Gromov.
\newblock {\em Metric structures for {R}iemannian and non-{R}iemannian spaces},
  volume 152 of {\em Progress in Mathematics}.
\newblock Birkh{\"a}user Boston, Inc., Boston, MA, 1999.
\newblock Based on the 1981 French original [ MR0682063 (85e:53051)], With
  appendices by M. Katz, P. Pansu and S. Semmes, Translated from the French by
  Sean Michael Bates.

\bibitem[Hal50]{MR0036767}
Marshall Hall, Jr.
\newblock A topology for free groups and related groups.
\newblock {\em Ann. of Math. (2)}, 52:127--139, 1950.

\bibitem[HO15]{Hull:2015fr}
Michael Hull and Denis Osin.
\newblock Transitivity degrees of countable groups and acylindrical
  hyperbolicity.
\newblock {\em to appear in Israel J. Math.}, 2015.

\bibitem[Kai97]{MR1485618}
Vadim~A. Kaimanovich.
\newblock Amenability, hyperfiniteness, and isoperimetric inequalities.
\newblock {\em C. R. Acad. Sci. Paris S{\'e}r. I Math.}, 325(9):999--1004,
  1997.

\bibitem[Kec10]{MR2583950}
Alexander~S. Kechris.
\newblock {\em Global aspects of ergodic group actions}, volume 160 of {\em
  Mathematical Surveys and Monographs}.
\newblock American Mathematical Society, Providence, RI, 2010.

\bibitem[KM04]{MR2095154}
Alexander~S. Kechris and Benjamin~D. Miller.
\newblock {\em Topics in orbit equivalence}, volume 1852 of {\em Lecture Notes
  in Mathematics}.
\newblock Springer-Verlag, Berlin, 2004.

\bibitem[KT10]{MR2599891}
John Kittrell and Todor Tsankov.
\newblock Topological properties of full groups.
\newblock {\em Ergodic Theory Dynam. Systems}, 30(2):525--545, 2010.

\bibitem[Lev95]{MR1366313}
Gilbert Levitt.
\newblock On the cost of generating an equivalence relation.
\newblock {\em Ergodic Theory Dynam. Systems}, 15(6):1173--1181, 1995.

\bibitem[LM14]{gentopergo}
Fran{\c c}ois Le~Ma{\^\i}tre.
\newblock The number of topological generators for full groups of ergodic
  equivalence relations.
\newblock {\em Invent. Math.}, 198:261--268, 2014.

\bibitem[LM15]{lm14nonerg}
Fran{\c c}ois Le~Ma{\^\i}tre.
\newblock On full groups of non-ergodic probability measure preserving
  equivalence relations.
\newblock {\em to appear in Ergodic Theory Dyn. Syst.}, 2015.

\bibitem[Zim84]{MR776417}
Robert~J. Zimmer.
\newblock {\em Ergodic theory and semisimple groups}, volume~81 of {\em
  Monographs in Mathematics}.
\newblock Birkh{\"a}user Verlag, Basel, 1984.

\end{thebibliography}

\end{document}